\documentclass[11pt]{amsart}
\usepackage{amssymb, mathrsfs, enumerate, amsmath, amsthm, url,  xcolor}
\usepackage[T1]{fontenc}
\usepackage{comment}
\usepackage{mathtools}
\usepackage{dsfont}
\usepackage[all]{xy}

\usepackage[utf8]{inputenc}
\usepackage[colorlinks=true,linkcolor=blue,citecolor=blue,urlcolor=blue]{hyperref}

\allowdisplaybreaks

\usepackage[english]{babel}

\usepackage[top=2.3cm, bottom=2.4cm, left=2.8cm, right=2.8cm]{geometry}

\theoremstyle{plain}
\newtheorem{theorem}{Theorem}[section]
\newtheorem{conjecture}[theorem]{Conjecture}
\newtheorem{proposition}[theorem]{Proposition}
\newtheorem{lemma}[theorem]{Lemma}
\newtheorem{corollary}[theorem]{Corollary}

\theoremstyle{definition}

\newtheorem{example}[theorem]{Example}

\newtheorem{remark}[theorem]{Remark}
\newtheorem{definition}[theorem]{Definition}
\newtheorem{construction}[theorem]{Construction}

\newtheorem{implementation}[theorem]{Implementation}

\newtheorem*{notation*}{Notation}

\newcommand{\SmallMatrix}[1]{\text{{
\Small
\arraycolsep=0.4\arraycolsep\ensuremath
    {\begin{pmatrix}#1\end{pmatrix}}}}}

\newcommand{\hs}{\kern 0.8pt}
\newcommand{\hssh}{\kern 1.2pt}
\newcommand{\hshs}{\kern 1.6pt}
\newcommand{\hssss}{\kern 2.0pt}
\newcommand{\ha}{{\kern 1pt}}

\newcommand{\hm}{\kern -0.8pt}
\newcommand{\hmm}{\kern -1.2pt}

\newcommand{\C}{{\mathds C}}
\newcommand{\R}{{\mathds R}}
\newcommand{\G}{{\mathds G}}

\newcommand{\GL}{{\rm GL}}
\newcommand{\Stab}{{\rm Stab}}

\newcommand{\Id}{{\rm Id}}

\newcommand{\upgam}{{{}^\gamma\hm}}

\newcommand{\ol}{\overline}
\newcommand{\Lie}{{\rm Lie}}
\newcommand{\PGL}{{\rm PGL}}
\newcommand{\Gr}{{\rm Gr}}

\newcommand{\into}{{\,\hookrightarrow\,}}
\newcommand{\onto}{{\,\twoheadrightarrow\,}}
\newcommand{\isoto}{{\,\overset\sim\longrightarrow\,}}

\newcommand{\z}{{Z}}
\renewcommand{\Lambda}{\textstyle\bigwedge}
\newcommand{\tftf}{{\mathfrak t}}

\newcommand{\Pf}{\operatorname{Pf}}
\newcommand{\im}{\operatorname{im}}
\newcommand{\rk}{\operatorname{rank}}

\newcommand{\Mat}{{\rm Mat}}
\newcommand{\Skew}{{\rm skew}}
\newcommand{\nd}{{\rm nd}}

\newcommand{\TT}{{\sf T}}

\newcommand{\dgam}{{\hs{\gamma}}}
\newcommand{\scdot}{\!\cdot\!}
\newcommand{\sdot}{{\kern-0.05pt\cdot\kern-0.05pt}}
\newcommand{\ii}{{i}}

\newcommand{\vk}{\varkappa}
\newcommand{\gl}{\mathfrak{gl}}

\title[Complex Lie algebra isomorphic to its complex conjugate]
{Constructing a complex Lie algebra isomorphic to its complex conjugate but not definable over reals}
\author[M. Borovoi]{Mikhail Borovoi}
\address{Borovoi: School of Mathematical Sciences,
Tel Aviv University, 6997801 Tel Aviv, Israel}
\email{borovoi@tauex.tau.ac.il}

\author[W. A. de Graaf]{Willem A. de Graaf}
\address{De Graaf: Universit\`a di Trento, Dipartimento di Matematica, Via Sommarive 14, 38123 Povo TN, Italy}
\email{willem.degraaf@unitn.it}

\author[R. M. Guralnick]{Robert M. Guralnick}
\address{Guralnick: Department of Mathematics, University of Southern California, Los Angeles, CA 90089-2532, USA}
\email{guralnic@usc.edu}

\thanks{Borovoi was partially supported by the Israel Science Foundation (grant 1030/22).}
\thanks{Guralnick was partially supported by Simons Foundation Fellowship 00019819.}

\keywords{Two-step nilpotent Lie algebra,  complex field, real field, real form.
}

\subjclass{
  12D99
, 17B30
}

\date{\today}

\begin{document}

\begin{abstract}
Following an idea of Demarche and using computer calculations of the authors
and ideas of the LLM Claude Fable 5,
we construct an explicit 10-dimensional two-step nilpotent complex Lie algebra
that is isomorphic to its complex conjugate
but cannot be defined over the field of real numbers $\R$.

\end{abstract}

\maketitle

\section{Introduction}

Let $k$ be a field of characteristic 0.
A {\em Lie algebra} $L$ over $k$ is a vector space over $k$ together with  a {\em Lie bracket},
that is, a $k$-bilinear map
\[ [\,,]\colon L\times L\to L, \quad x,y\mapsto [x,y]\]
satisfying the following conditions (see \cite{Jacobson}):
\begin{align}
&[x,y]=-[y,x]\quad\ \text{for all}\ \,x,y\in L;\label{e:xy}\\
&[x,[y, z]\hs]+[y,[z,x]\hs] + [z,[x,y]\hs] =0 \quad\ \text{for all}\ \,x,y,z\in L.\label{e:Jacobi}
\end{align}
A {\em two-step nilpotent Lie algebra} $L$ over $k$ is a  vector space  over $k$ together with a $k$-bilinear map
$ [\,,]\colon L\times L\to L$ satisfying \eqref{e:xy} and the following condition:
\begin{equation}
[x,[y,z]\hs]=0 \quad\ \text{for all}\ \,x,y,z\in L.\label{e:2-step}
\end{equation}
Clearly, \eqref{e:2-step} implies \eqref{e:Jacobi},
and so any two-step nilpotent Lie algebra is a Lie algebra.

Let $L$ be a Lie algebra over $k$. The {\em center} of $L$ is the $k$-subspace
\[ Z(L)=\{x\in L\ \,\big|\ \,[x,y]=0 \ \, \text{for all}\ y\in L\}.\]
A Lie algebra $L$ is two-step nilpotent if and only if
$[\hs[L,L],L]=0$, that is, $[L,L]\subseteq Z(L)$.
We say that a two-step nilpotent Lie algebra is {\em non-degenerate} if
\begin{equation*} 
[L,L]=Z(L).
\end{equation*}

In this paper, all vector spaces and Lie algebras are finite-dimensional.
Let $L$ be a Lie algebra of dimension $d$ over $k$, and let $e_1,\dots,e_d$ be a basis of $L$.
Then we may write
\[[e_p,e_q]=\sum_{r=1}^d c_{p,q}^r e_r\quad \ \text{for}\ \, 1\le p<q \le d\quad\text{and some}\ \, c_{p,q}^r\in k. \]
The elements $c_{p,q}^r$ are called the {\em structure constants} of $L$
with respect to the basis $e_1,\dots,e_d$\hs.
They uniquely determine the isomorphism class of $L$.

We will consider {\em complex Lie algebras}, that is, Lie algebras over the field $\C$ of complex numbers,
and {\em real Lie algebras}, that is, Lie algebras over  the field $\R$ of real numbers.
By the {\em complex conjugate} of a finite-dimensional complex Lie algebra $L$ 
we mean the complex Lie algebra obtained from $L$
by applying complex conjugation to the structure constants of $L$ (with respect to a chosen basis).
See Section \ref{s:2-C} for a construction not using a choice of a basis.

Let $L_\R$ be a finite-dimensional {\em real} Lie algebra.
Then for the complex Lie algebra $L=L_\R\otimes_\R\C$
there is a basis with respect to which all structure constants are real,
and therefore $L$ is isomorphic to its complex conjugate.
Jonas Der\'e proposed the following conjecture:

\begin{conjecture}[{\cite[Conjecture 1 on page 494]{Dere}}]
\label{cj:Dere}
Assume that a finite-dimensional complex Lie algebra $L$ is isomorphic to its complex conjugate. Then
$L$ can be defined over $\R$, that is, $L\simeq L_\R\otimes_\R\C$ for some real Lie algebra $L_\R$\hs.
\end{conjecture}

Such a real Lie algebra $L_\R$ is called a {\em real form} of $L$.
Thus Conjecture \ref{cj:Dere} says that if $L$ is isomorphic to its complex conjugate,
then $L$ {\em admits a real form}, that is, there exists a real form of $L$.

Cyril Demarche \cite[Theorem 1]{Demarche} disproved this conjecture.
He proved that {\em there exists} a $10$-dimensional two-step nilpotent complex Lie algebra $L$ 
that is isomorphic to its complex conjugate
but admits no real form. However, Demarche did not produce an explicit example of such a Lie algebra $L$.
In this paper we produce such an example.

\begin{theorem}\label{t:main}
Let $V_{10}=\C^{10}$ with the standard basis $e_1,\dots,e_{10}$\hs,
and let $V_{10}^*$ be the dual space with the dual basis $e_1^*,\dots, e_{10}^*$.
Let $\Lambda^2 V_{10}^*$ denote the space of skew-symmetric bilinear forms on $V_{10}$.
Consider the complex subspaces $V_6=\langle e_1,\dots, e_6\rangle$ and $V_4=\langle e_7,\dots, e_{10}\rangle$.
Consider the following tensors in $\Lambda^2 V_6^*$  (skew-symmetric bilinear forms on $V_6$):
\begin{equation}\label{e:WW}
\begin{aligned}
&w_1= e_4^*\wedge e_5^*+ e_3^*\wedge e_6^*\hs,\\
&w_2=-e_4^*\wedge e_6^*-e_3^*\wedge e_5^*+e_2^*\wedge e_6^*
+e_2^*\wedge e_5^*-e_1^*\wedge e_6^*+e_1^*\wedge e_5^*\hs,\\
&w_3=-e_4^*\wedge e_6^*-e_3^*\wedge e_5^*+e_2^*\wedge e_5^*
-e_2^*\wedge e_3^*-e_1^*\wedge e_6^*+e_1^*\wedge e_4^*\hs,\\
&w_4=-\tfrac{i}{2}\,e_2^*\wedge e_4^*+\tfrac{i}{2}\,e_1^*\wedge e_3^*
+e_1^*\wedge e_2^*\hs,
\end{aligned}
\end{equation}
where $i^2=-1$.
Consider the following tensor in $(\Lambda^2 V_6^*)\otimes V_4\subset (\Lambda^2 V_{10}^*)\otimes V_{10}$:
\begin{equation}\label{e:tf}
\tftf=\tftf(w_1,w_2,w_3,w_4)= w_1\otimes e_7+ w_2\otimes e_8+w_3\otimes e_9+w_4\otimes e_{10}\hs.
\end{equation}
Consider the Lie bracket on $V_{10}$ induced by $\tftf$:
\[ [\, ,]_\tftf\colon V_{10}\times V_{10}\to V_{10},
   \quad\ [x,y]=\sum_{\ell=1}^4 w_\ell(x,y)\hs e_{6+\ell}\hs, \]
where we regard the bilinear forms $w_\ell$ on $V_6$ as bilinear forms on $V_{10}=V_6\oplus V_4$
via the projection $V_{10}\to V_6$ with kernel $V_4$.
We obtain a 10-dimensional two-step nilpotent Lie algebra $L=L_\tftf\coloneqq(V_{10},[\,,]_\tftf)$\hs.
Then  (i)~$L$ is isomorphic to its complex conjugate, but (ii)~$L$ admits no real form.
\end{theorem}

Here (i) can be easily checked by an explicit calculation, and (ii) is  proved
in the paper, following an idea of Demarche
and using computer calculations of the authors based on ideas of the LLM (Large Language Model)  Claude Fable 5.
Our $10$-dimensional Lie algebra $L=L_\tftf$ is a  two-step nilpotent Lie algebra.
Indeed, it is clear from \eqref{e:WW} and \eqref{e:tf} that
\[ [L,L]\subseteq V_4\subseteq Z(L).\]
Moreover, it is easy to check using a computer that
\[ [L,L]= V_4=Z(L),\]
that is, $L$ is a {\em non-degenerate} two-step nilpotent Lie algebra.

See Construction \ref{cons:W}  below for our method
of constructing the set of tensors \eqref{e:WW}.
Moreover, in Remark \ref{cons:W'} we describe a method permitting one to construct
other sets of tensors $w_1',\dots,w_4'\in \Lambda^2 V_6^*$
such that Theorem \ref{t:main} holds for $L'=L_{\tftf'}$ with $\tftf'=\tftf(w_1',w_2',w_3', w_4')$
as in \eqref{e:tf}.

\begin{remark}
It follows from \cite[Tables 2 and 8]{GT} (see also \cite{BDG} for dimension 8 only)
that in every complex two-step nilpotent Lie algebra
of dimension $\le 8$, there exists a basis in which all structure constants are either 0 or 1,
and so the Lie algebra can be defined over $\R$ and is isomorphic to its complex conjugate.
In dimension 9, there exist complex two-step nilpotent Lie algebras $L$
such that $L$ is not isomorphic to its complex conjugate, but computer computations 
of the second-named author
show that for complex two-step nilpotent Lie algebras of dimension 9,
Conjecture \ref{cj:Dere} holds.
Thus $d=10$ is the lowest dimension $d$ such that
for complex two-step nilpotent Lie algebras of dimension $d$,
Conjecture \ref{cj:Dere} fails.
By Theorem \ref{t:m} below, Conjecture \ref{cj:Dere} fails also for all $d>10$.
\end{remark}

\begin{theorem}\label{t:m} 
Let $L$ be the 10-dimensional non-degenerate two-step nilpotent complex Lie algebra 
constructed in Theorem \ref{t:main}.
For any integer $m>0$, consider the {\em degenerate} 
two-step nilpotent complex Lie algebra $L'=L\times \C^m$ of dimension $10+m$,  
where we regard $\C^m$ as an $m$-dimensional abelian complex Lie algebra.
Then (i) $L'$ is isomorphic to its complex conjugate,  but (ii) $L'$ admits no real form.
\end{theorem}

Let $U$ be a complex linear algebraic group. 
We write $\dgam U$ for the complex conjugate of $U$,
that is, for the complex linear algebraic group obtained from $U$ 
by transport of scalars by the complex conjugation  $\gamma\in{\rm Gal}(\C/\R)$.

By a {\em real form} of $U$ we mean a real algebraic group $U_\R$
such that $U\simeq  U_\R\times_\R \C$.
If $U$ admits a real form, then $\dgam U\simeq U$.
Below we construct an {\em explicit example} of a complex algebraic group $U$
such that $\dgam U\simeq U$ but $U$ admits no real form;
cf. \cite[Corollary 7]{Demarche} where the {\em existence} of such $U$ was proved.

\begin{example}\label{ex:U}
For the Lie algebra $L$ of Theorem \ref{t:main},
let $v\oplus w\in V_6\oplus V_4=V_{10}=L$.
Denote by $\Phi(v)$ the $4\times 6$-matrix such that
\[\Phi(v)\cdot v'=\frac12 [v,v']\]
where $v,v'\in V_6\subset L$ are regarded as {\em column} vectors.
Let $\rho(v\oplus w)$ denote the $11\times 11$  block matrix
\[
\begin{pmatrix}
0_{4\times 4} &\Phi(v)       &w\\
0             &0_{6\times 6} &v\\
0             &0             &0
\end{pmatrix}.
\]
One checks that 
\[ \big[\hs\rho(v\oplus w),\hs\rho(v'\oplus w')\hs\big]=\rho\big(\hs0\oplus[v,v']\hs\big)
     =\rho\big(\hs[v\oplus w, v'\oplus w']\hs\big),\]
that is, $\rho$ is a faithful 11-dimensional representation 
of our 10-dimensional Lie algebra $L$.
Moreover, one can check that for the matrix $\rho(v\oplus w)$ we have $\rho(v\oplus w)^2=0$,
whence  $\exp \rho(v\oplus w)=1 +\rho(v\oplus w)$
and  $\exp \rho(L)$ is the set of matrices of the form
\begin{equation}\label{e:U}
\begin{pmatrix}
I_4 &\Phi(v) &w\\
0 & I_6 &v\\
0 &0 & 1
\end{pmatrix} \quad\ \text{for} \ \,v\oplus w\in V_6\oplus V_4=V_{10}=L.
\end{equation}
This is a subgroup of $\GL(11,\C)$ with the multiplication law 
\begin{equation*}
\begin{pmatrix}
I_4 &\Phi(v) &w\\
0 & I_6 &v\\
0 &0 & 1
\end{pmatrix}
\begin{pmatrix}
I_4 &\Phi(v') &w'\\
0 & I_6 &v'\\
0 &0 & 1
\end{pmatrix}
=
\begin{pmatrix}
I_4 &\Phi(v+v') &w+w'+\frac12 [v,v']\\
0 & I_6 &v+v'\\
0 &0 & 1
\end{pmatrix}.
\end{equation*}

Let $U\subset \GL_{11,\C}$ be the unipotent complex algebraic group 
whose group of $\C$-points is given by \eqref{e:U}.
Then for the Lie algebra of $U$ we have $\Lie\, U=\rho(L)\simeq L$,
and by the classical correspondence between nilpotent Lie algebras and unipotent algebraic groups
(\hs see \cite[Theorem 14.37]{Milne}\hs) we obtain from Theorem \ref{t:main} the following corollary. 
\end{example}

\begin{corollary}
The 10-dimensional unipotent complex algebraic group $U$  of Example \ref{ex:U} 
is isomorphic to its complex conjugate but admits no real form.
This group is a central extension of $(\G_{a,\C})^6$ by $(\G_{a,\C})^4$,
where $\G_{a,\C}$ denotes the additive group over $\C$.
\end{corollary}

For an integer $m>0$, set $U'=U\times_\C (\G_{a,\C})^m$.
Then $\Lie\,U'\simeq L'\coloneqq L\times \C^m$, and 
from Theorem \ref{t:m} we obtain the following corollary.

\begin{corollary}
For any integer $m>0$, the unipotent complex algebraic group 
$U'=U\times_\C(\G_{a,\C})^m$ of dimension $10+m$
is isomorphic to its complex conjugate but admits no real form.
\end{corollary}

The following two remarks are due to Boris Kunyavski\u{\i} (private communication).

\begin{remark}[Associative algebras]
\label{r:K1}
For a two-step nilpotent complex Lie algebra $L$,
consider the  complex algebra without unit $A_L=L$ with multiplication 
\[x\cdot y=[x,y].\]
The algebra $A_L$ is clearly associative, and we have $x^2=0$ for any $x\in A_L$.
One can easily see that a real form $(A_L)_\R$ of $A_L$, if it exists, corresponds to a real 
form $L_\R$ of $L$. Now taking for $L$ the Lie algebra of Theorem \ref{t:main},
we obtain an example of a 10-dimensional associative algebra without unit 
that is isomorphic to its complex conjugate but admits no real form.
\end{remark}

\begin{remark}[Associative algebras with unit]
\label{r:K2}
For the associative complex algebra $A_L$ (without unit) of Remark \ref{r:K1},
consider the unitized algebra $U_L=A_L\oplus\C\sdot1$ with the multiplication law
\[(a_1+\lambda_1 \sdot1)(a_2+\lambda_2\hs \sdot1)
   =a_1 a_2 +\lambda_1 a_2 +\lambda_2 a_1 +\lambda_1\lambda_2\hs\sdot 1
   \quad\ \text{for}\ \, a_1,a_2\in A_L,\ \lambda_1,\lambda_2\in\C.\]
Then $1$ is the unit of $U_L$. 
Consider the Jacobson radical $J(U_L)$
(the intersection of all maximal left ideals of $U_L$);
then  $J(U_L)=A_L$.
One can easily see that for a real form $(U_L)_\R$ of $U_L$, if it exists,
its Jacobson radical $J((U_L)_\R)$ satisfies 
\[ J\big((U_L)_\R\big)\otimes_\R\C=J(U_L)=A_L\hs,\]
and therefore the associative algebra $J\big((U_L)_\R\big)$ is a real form of $A_L$.
Now we take for $L$ the 10-dimensional two-step nilpotent 
complex Lie algebra of Theorem \ref{t:main}.
Then by Remark \ref{r:K1} the associative algebra  $A_L$ 
is isomorphic to its complex conjugate but admits no real form,  
and therefore the 11-dimensional {\em associative algebra with unit} $U_L$
is isomorphic to its complex conjugate but admits no real form.
\end{remark}

\subsection*{Assumptions and notation} 
In this paper, all vector spaces and Lie algebras are finite-dimensional
over fields of characteristic 0, and starting from Section \ref{s:2-C}, over $\C$.
By a {\em real form} of a complex Lie algebra $L$
we mean a real Lie algebra $L_\R$ such that $L_\R\otimes_\R \C\simeq L$.
We denote by $V^*$ the dual space to a vector space $V$, by $\Lambda^2 V$ the second exterior power of $V$,
and by $\Lambda^2 V^*\cong\big(\Lambda^2 V\big)^*$ the space of skew-symmetric bilinear forms on $V$.

\subsection*{Plan of the paper}
The plan of the rest of this paper is as follows. 
Let $L$ be a non-degenerate two-step nilpotent Lie algebra
over a field $k$ of characteristic 0.
Write $V=L/Z$ where $Z=Z(L)$;
then we have a natural skew-symmetric  bilinear map $V\times V\to Z$, 
inducing a  linear map $\psi_L\colon \Lambda^2 V\onto Z$,
which is surjective because $L$ is non-degenerate.
Set $F=\ker\psi_L\subset \Lambda^2 V$; then the pair $(V,F)$ is non-degenerate 
in the sense of Definition \ref{d:non-deg-F}.
Consider the dual space $\Lambda^2 V^*$ to $\Lambda^2 V$
and let $W=F^\bot\subset \Lambda^2 V^*$, the orthogonal complement to $F$.
Then $(V,W)$ is a $*${\em-pair}, that is, 
a pair where $V$ is a vector space and $W\subset\Lambda^2 V^*$.
This $*$-pair is non-degenerate in the sense of Definition \ref{d:*-pair}.
In Section \ref{s:V-F} we construct 
a non-degenerate two-step nilpotent  Lie algebra $L=L(V,F)$ 
from a non-degenerate pair $(V,F)$ with $F\subset \Lambda^2 V$, or equivalently, 
from the corresponding non-degenerate $*$-pair $(V,W)$.

In Section \ref{s:2-C}, for $k=\C$ and  a non-degenerate $*$-pair $(\C^n,W)$, 
we give a necessary and sufficient condition in terms of $W$
for the existence of a real form $L_\R$ of $L=L(\C^n,W)$.
In Section \ref{s:F} we prove Theorem \ref{t:main} 
modulo Theorem \ref{t:Stab} 
which asserts that for the 4-dimensional subspace $W\subset\Lambda^2 V_6^*$
generated by the tensors $w_1,\dots,w_4$ of \eqref{e:WW},
the stabilizer in $\PGL(6,\C)$
of $W\in \Gr(4,\Lambda^2 V_6^*)$ is trivial, 
where $\Gr(4,\Lambda^2 V_6^*)$ is the Grassmannian 
of 4-dimensional subspaces in $\Lambda^2 V_6^*$.
After that, we deduce Theorem \ref{t:m} from Theorem \ref{t:main}.
In Section \ref{s:Stab-W} we prove Theorem \ref{t:Stab} using computer calculations.
In the arXiv version of this paper, in Appendix \ref{app:code}
we provide the program code used in computer calculations.

\section{Constructing a two-step nilpotent Lie algebra}
\label{s:V-F}

Let $k$ be a field of characteristic 0.
In this section, all vector spaces and Lie algebras are defined over $k$.
The following two constructions are due to Demarche \cite{Demarche}.

\begin{construction}\label{con:L-to-VF}
Let $L$ be a two-step nilpotent Lie algebra.
Write $Z=Z(L)$ and set  $V_L=L/Z$.
Then we have a canonical skew-symmetric bilinear map
\begin{equation}\label{e:L-L-Z}
L\times L\to Z, \quad\ (x,y)\mapsto [x,y]\in Z,
\end{equation}
which induces an embedding
\begin{equation}\label{e:[L,L]-into-Z} 
[L,L]\into Z
\end{equation}
where by definition $[L,L]$ is the subspace of $L$ generated by the brackets 
$[x,y]$ with $x,y\in L$.


Let $L$ be a non-degenerate two-step nilpotent Lie algebra.
The map \eqref{e:L-L-Z} clearly induces a  skew-symmetric bilinear map
\[V_L\times V_L \to Z\]
and a {\em linear} map
\[\psi_L\colon \Lambda^2 V_L\to Z,\]
which is surjective because $L$ is non-degenerate.
Set
\[ F_L=\ker \psi_L\subset\Lambda^2 V_L\hs;\]
then
\begin{equation}\label{e:Z-Lambda2}
Z\cong \big(\Lambda^2 V_L\big)/F_L.
\end{equation}
Thus from a non-degenerate two-step nilpotent Lie algebra $L$ we obtain a pair $(V_L, F_L)$.
\end{construction}

\begin{definition}\label{d:non-deg-F}
Let $V$ be a vector space over $k$, and $F\subset \Lambda^2 V$ be a $k$-subspace.
We say that the pair $(V,F)$ is  {\em non-degenerate} if 
\begin{equation}\label{e:non-deg-F}
\text{for any nonzero $v\in V$ there exists $v'\in V$ with $v\wedge v'\notin F$.}
\end{equation}
\end{definition}

Observe that the pair $(V_L, F_L)$ obtained from a non-degenerate Lie algebra $L$ as above is non-degenerate.
Indeed, assume that $v\in V_L$ is such that  $v\wedge v'\in F_L$ for all $v'\in V_L$,
and let $x\in L$ be a preimage of $v$.
Then $[x,x']=0$ for all $x'\in L$, whence $x\in Z$ and $v=0$.

\begin{construction}\label{con:VF-to-L}
Let $V$ be a finite-dimensional vector space over $k$,
and let $F\subset \Lambda^2 V$ be a  linear subspace
such that the pair $(V,F)$ is non-degenerate.
Consider the Lie algebra $L(V,F)$  with underlying vector space
$V\oplus (\Lambda^2 V)/F$ and with Lie bracket
\begin{equation}\label{e:Lie-bracket}
\big[ (v_1,x_1), (v_2,x_2)\big]=(\hs 0,\, v_1\wedge v_2+F\hs)
\end{equation}
where $v_1,v_2\in V$,  $x_1,x_2\in (\Lambda^2 V)/F$, and  $v_1\wedge v_2+F\in (\Lambda^2 V)/F$.
Then for all $y_1,y_2,y_3\in L(V,F)$ we have
\begin{align*}
& [y_1,y_2]=-[y_2,y_1],  \\
& [\,[y_1,y_2],y_3]=0.
\end{align*}
Thus $L(V,F)$ is a two-step nilpotent Lie algebra.
\end{construction}

\begin{lemma}\label{l:ZZ}
For a non-degenerate pair $(V,F)$ as in Definition \ref{d:non-deg-F}, consider 
the  two-step nilpotent Lie algebra $L(V,F)$ as in Construction \ref{con:VF-to-L}.
Write  $Z= \z\big(L(V,F)\big)$.
Then: 
\begin{enumerate}
\item[\rm(i)] $Z=0\oplus(\Lambda^2 V)/F$;
\item[\rm(ii)] The two-step nilpotent Lie algebra $L=L(V,F)$ is non-degenerate,
that is, $[L,L]=Z$.
\end{enumerate}
\end{lemma}

\begin{proof} 
It follows from \eqref{e:Lie-bracket} that 
\[0\oplus(\Lambda^2 V)/F\hs\subseteq \hs Z\big(L(V,F)\big).\]
We show that this inclusion is an equality.
Indeed, for $(v,x)\in V\oplus (\Lambda^2 V)/F$,
assume that $(v,x)\in Z\big(L(V,F)\big)$.
Then 
\[ \big[(v,x), (v',0)\big]=0\quad\ \text{for all }\ \, v'\in V.\]
By \eqref{e:Lie-bracket}, this means that 
\begin{equation}\label{e:v-v'}
v\wedge v'\in F\quad\ \text{for all }\ \, v'\in V.
\end{equation}
Since the pair $(V,F)$ is non-degenerate, 
from \eqref{e:v-v'} and \eqref{e:non-deg-F} we obtain that $v=0$.
Hence $(v,x)\in 0\oplus(\Lambda^2 V)/F$, which proves (i).

Since $\Lambda^2 V$ is generated by the bivectors of the form $v_1\wedge v_2$,
we see from \eqref{e:Lie-bracket} that $[L,L]=0\oplus \Lambda^2 V/F$.
Now from (i) we obtain that $[L,L]=Z$, that is, (ii) holds.
\end{proof}

\begin{proposition}\label{p:bij}
Constructions \ref{con:L-to-VF} and \ref{con:VF-to-L} induce a canonical bijection
between the set of isomorphism classes of
non-degenerate two-step nilpotent Lie algebras over $k$
and the set of isomorphism classes of non-degenerate
pairs $(V,F)$ over $k$ as in Construction \ref{con:VF-to-L}.
\end{proposition}

\begin{proof}
Clearly, these constructions induce maps of isomorphism classes.
We show that these maps are mutually inverse.

If we start with a non-degenerate pair $(V,F)$ as in Construction \ref{con:VF-to-L}, then by Lemma \ref{l:ZZ}(ii), the two-step nilpotent Lie algebra
$L(V,F)$ is non-degenerate, and clearly we have
\begin{equation*}
(\hs V_{L(V,F)},\hs F_{L(V,F)}\hs)=(V,F).
\end{equation*}

Let us now start with a non-degenerate two-step nilpotent Lie algebra $L$;
then $(V_L,F_L)$ is a non-degenerate pair.
Consider the  two-step nilpotent Lie algebra
\[ V_L\oplus \Lambda^2 V_L/F_L\hs.\]
Choose a splitting
\[\varsigma\colon V_L\to L\]
of the projection homomorphism of vector spaces
\[ L\to L/Z=V_L\hs.\]
We set
\[\tilde V_L=\varsigma(V_L)\subset L,\quad\  \tilde F_L=\varsigma_*(F_L)\subset \Lambda^2\tilde V_L\hs.\]
Then  using  \eqref{e:Z-Lambda2} we obtain
\[ L=\tilde V_L\oplus  Z\cong \tilde V_L\oplus \Lambda^2 V_L/F_L
\cong\tilde V_L\oplus \Lambda^2\tilde V_L/\tilde F_L\hs.\]
We have an isomorphism $\varsigma\colon V_L\isoto \tilde V_L$\,,
and the induced  isomorphism $\Lambda^2 V_L\isoto \Lambda^2 \tilde V_L$
sends $F_L$ to $\tilde F_L$.
Thus, after choosing a splitting $\varsigma$, we obtain an isomorphism
\[ L\simeq \tilde V_L\oplus \Lambda^2\tilde V_L/\tilde F_L\simeq V_L\oplus \Lambda^2 V_L/F_L=L(V_L,F_L).\qedhere\]
\end{proof}

\begin{corollary}
Let  $(V,F)$ and $(V',F')$ be two non-degenerate pairs 
as in Construction \ref{con:VF-to-L}.
Then  the non-degenerate two-step nilpotent Lie algebras 
$L(V,F)$ and $L(V',F')$ are isomorphic
if and only if $(V,F)$ and $(V',F')$ are isomorphic,
that is, there exists an isomorphism
$\varphi\colon V\isoto V'$ such that $\varphi_*(F)=F'$.
\end{corollary}

\begin{corollary}\label{c:same-V}
Let $(V,F_1)$ and $(V,F_2)$ be two non-degenerate pairs with the same $V$.
Then  the non-degenerate two-step nilpotent Lie algebras 
$L(V,F_1)$ and $L(V,F_2)$ are isomorphic
if and only if there exists $g\in\GL(V)$
such that $g_*(F_1)=F_2$.
\end{corollary}

We consider the dual space $V^*$ to $V$. The natural pairing $V^*\times V\to k$ induces a pairing
\[ \Lambda^2 V^*\times \Lambda^2 V\to k,\]
which permits us to identify $\Lambda^2 V^* =(\Lambda^2 V)^*$.

\begin{definition}\label{d:*-pair}
By a $*$-pair we mean a pair  $(V, W)$ where $V$ is a vector space over $k$
and  $W\subseteq \Lambda^2 V^*$  is a $k$-subspace.
We say that a $*$-pair $(V,W)$ is {\em non-degenerate}
if for any $v\neq 0\in V$ there exists $v'\in V$ such that 
$\langle W, v\wedge v'\rangle \neq 0$.
\end{definition}

For a pair $(V,F)$ as above, consider the orthogonal complement
\[ W=F^\bot\subseteq\Lambda^2 V^*.\]
We obtain a $*$-pair $(V, W)$.
Conversely, from a $*$-pair $(V,W)$ we obtain a pair $(V,F)$ 
with $F=W^\bot\subset \Lambda^2 V$.

\begin{lemma}\label{l:ndp}
A pair $(V,F)$ as above is non-degenerate if and only if the corresponding $*$-pair $(V,W)$  with $W=F^\bot$ is non-degenerate.
\end{lemma}

\begin{proof}
Indeed, since $W=F^\bot$, the condition $v\wedge v'\notin F$ 
of Definition \ref{d:non-deg-F} is clearly equivalent 
to the condition $\langle W,v\wedge v'\rangle\neq 0$ of Definition \ref{d:*-pair}.
\end{proof}

The group $\GL(V)$ naturally acts on $\Lambda^2 V$ and $\Lambda^2 V^*$.
Since both actions are the natural ones, the pairing is invariant, that is,
\[\langle g^*w, g_*\xi\rangle=\langle w,\xi\rangle 
\quad\ \text{for}\ \, w\in \Lambda^2 V^*,\ \xi\in \Lambda^2 V,\  g\in\GL(V),\]
whence
\[g^*(W_1)=W_2\, \Longleftrightarrow\, g_*(W_1^\bot)=W_2^\bot.\]

\begin{corollary}
Let $(V,F_1)$ and $(V,F_2)$ be two non-degenerate pairs with the same $V$,
and let $(V,W_1)$ and $(V,W_2)$ be the corresponding $*$-pairs.
Then the non-degenerate two-step nilpotent Lie algebras $L(V,F_1)$ and $L(V,F_2)$ are isomorphic
if and only if there exists $g\in\GL(V)$
such that $g^*(W_1)=W_2$.
\end{corollary}

\begin{proposition}\label{p:nd}
Let $L$ be a finite-dimensional two-step nilpotent Lie algebra, possibly degenerate, over a field $k$ of characteristic 0. 
Write $Z=Z(L)$ for the center of $L$. Then
\begin{enumerate}
\item[\rm(i)]There exists a decomposition $L\simeq L_\nd\times k^m$, where $L_\nd$ is a  
{\em non-degenerate} two-step nilpotent Lie algebra, 
$m=\dim Z-\dim [L,L]$, 
and $k^m$ is regarded as an abelian Lie algebra of dimension $m\ge 0$.
\item[\rm(ii)] The isomorphism class of $L_\nd$ does not depend on the choice of decomposition:
if $L\simeq L'_\nd\times k^{m'}$ with  $L'_\nd$ non-degenerate two-step nilpotent and $k^{m'}$ abelian,
then $L'_\nd\simeq L_\nd$ and $m'=m$.
\end{enumerate}
\end{proposition}

\begin{proof}
We prove (i).
Let $\tilde V$ be a complement of $Z$ in $L$, that is,
\begin{equation*}
L=\tilde V\oplus Z.
\end{equation*}
We have $[L, L]\subseteq Z$, because the Lie algebra $L$ is two-step nilpotent.
Set
\[ L_\nd=\tilde V\oplus [L,L]\,\subseteq\, \tilde V\oplus Z = L.\]
Since 
$ [\tilde V,\tilde V]= [L, L]\subseteq Z$, 
we see that  $L_\nd$ is a Lie subalgebra of $L$,
which is clearly two-step nilpotent.

We show that $L_\nd$ is non-degenerate.
We have $[L_\nd,L_\nd]=[\tilde V, \tilde V]=[L,L]\subseteq Z(L_\nd)$.
Conversely, let $v+d\in Z(L_\nd)$ with $v\in\tilde V,\ d\in [L,L]\subseteq Z$.
Then $[(v+d),\tilde V]=0$, whence $[v,\tilde V]=0$.
On the other hand, clearly $[v,Z]=0$.
Since $L=\tilde V\oplus Z$, we obtain that $v\in Z$ and that $v\in \tilde V\cap Z=0$,
whence $v+d=d\in [L,L]$, and therefore $Z(L_\nd)\subseteq [L,L]=[L_\nd,L_\nd]$.
Thus $Z(L_\nd)=[L_\nd,L_\nd]$ and so $L_\nd$ is indeed non-degenerate.

Let $C$ be a complement of $[L,L]$ in $Z$, that is, 
\[ Z=[L,L]\oplus C.\]
We regard $C\simeq k^m$ as an abelian Lie algebra of dimension $m$.
Then 
\[L=L_\nd\times C \simeq L_\nd\times k^m,\]
which proves (i).

We prove (ii).
Let $\tilde V_\nd$ be a complement of  $Z(L_\nd)$ in $L_\nd$;
then it is also a complement of $Z(L)$ in $L$, and we have 
\[L_\nd=\tilde V_\nd\oplus Z(L_\nd)= \tilde V_\nd\oplus [L_\nd,L_\nd]=\tilde V_\nd\oplus [L,L].\]
Let $L\simeq L'_\nd\times k^{m'}$ be another such decomposition
and let $\tilde V'_\nd$ be a complement of  $Z(L'_\nd)$ in $L'_\nd$;
then it is also a complement of $Z(L)$ in $L$,
and we have
\[L'_\nd= \tilde V'_\nd\oplus Z(L'_\nd)=\tilde V'_\nd\oplus [L'_\nd,L'_\nd]=\tilde V'_\nd\oplus [L,L].\]
The isomorphism $\sigma\colon\tilde V_\nd\isoto L/Z(L)\isoto \tilde V'_\nd$
induces an isomorphism of vector spaces  
\[\tilde\sigma=\sigma\oplus {\rm Id}_{[L,L]}\colon L_\nd\isoto L'_\nd\hs,\]
whence $m=m'$.
We show that the isomorphism of vector spaces $\tilde\sigma$ respects brackets.
Indeed, $\sigma(v)\equiv v \pmod{Z(L)}$ for $v\in \tilde V_\nd$\hs, 
whence $[\sigma(v_1),\sigma(v_2)]=[v_1,v_2]$ for all $v_1,v_2\in\tilde V_\nd$\hs.
Thus $\tilde\sigma$ is an isomorphism of Lie algebras, which proves (ii).
\end{proof}

\begin{definition}
Let $L$ be a Lie algebra over a field $K$ of characteristic 0, and let $k\subset K$ be a subfield.
A {\em $k$-form} of $L$ is a Lie algebra $L_k$ over $k$ such that $L_k\otimes_k K\simeq L$.
We say that $L$ {\em admits a $k$-form} if there exists a $k$-form of $L$.
\end{definition}

\begin{proposition}\label{p:dg-R-C}
Let $K/k$ be an arbitrary extension of fields of characteristic $0$, and let 
$L$ be a non-degenerate two-step nilpotent Lie algebra over $K$.
Consider the degenerate two-step nilpotent Lie algebra $L'=L\times K^m$ over $K$
where $m>0$. 
Then $L'$ admits a $k$-form if and only if $L$ admits a $k$-form.
\end{proposition}

\begin{proof}
Suppose that $L$ admits a $k$-form $L_k$. 
Then clearly $L_k\times k^m$ is a $k$-form of $L\times K^m=L'$.

Conversely, suppose that $L'$ admits a $k$-form $L'_k$.
By Proposition \ref{p:nd}(i) there exists a decomposition
\begin{equation}\label{e:nd-R}
L'_k\simeq L'_{k,\nd}\times k^{m'}
\end{equation}
where $L'_{k,\nd}$ is a {\em non-degenerate} two-step nilpotent  Lie algebra over $k$.
From \eqref{e:nd-R} we obtain a decomposition 
\begin{equation*}
L'\simeq (L'_{k,\nd}\otimes_k  K) \times K^{m'},
\end{equation*}
where the two-step nilpotent  Lie algebra $L'_{k,\nd}\otimes_k K$ over $K$ is non-degenerate
because 
\[ Z(L'_{k,\nd}\otimes_k\hm K)=Z(L'_{k,\nd})\otimes_k\hm K=
\big[L'_{k,\nd},L'_{k,\nd}\big] \otimes_k\hm K=
\big[L'_{k,\nd}\otimes_k\hm K\hs,\hs L'_{k,\nd}\otimes_k\hm K\big].\]
Since also $L'=L\times K^m$
where $L$ is a non-degenerate two-step nilpotent Lie algebra over $K$,
Proposition \ref{p:nd}(ii) implies that $L'_{k,\nd}\otimes_k K\simeq L$.
Thus $L'_{k,\nd}$ is a  $k$-form of $L$, as desired.
\end{proof}

\section{Two-step nilpotent Lie algebras over $\C$}
\label{s:2-C}

Now we take $k=\C$. Write ${\rm Gal}(\C/\R)=\{1,\gamma\}$ 
where $\gamma$ denotes the complex conjugation.
When $a\in \C$ is a complex number, we denote by $\upgam a$ 
the complex conjugate number.
For a Lie algebra $L$ over $\C$, let $\dgam L$
denote the Lie algebra over $\C$ obtained by transporting structure via $\gamma$.
This means that $\dgam L$ is the same set $L$
with the same abelian group structure and the same Lie bracket,
but with the new multiplication by scalars
\[ a*x\coloneqq\upgam a\cdot x\quad\ \text{for all}\ \,a\in \C,\ x\in L.\]
For a vector space $V$ over $\C$, we define $\dgam V$ similarly.

Let $(V,F)$ be a non-degenerate pair as in Construction \ref{con:VF-to-L}.
Then
\begin{equation}\label{e:dgam}
\dgam L(V,F)\cong L(\hs\dgam V, \dgam F\hs).
\end{equation}

Let $V'=V'_\R\otimes_\R\C$ for some real vector space $V'_\R$; then we identify $\dgam V'=V'$,
and consider the natural action $x\mapsto \upgam x$ of   $\gamma$ on $V'$.
For a $\C$-subspace $W'\subset V'$, write $\upgam\hs W'=\{\upgam x \ |\ x\in W'\}$.
Then
\begin{equation}\label{e:upgam}
(\dgam V',\dgam W')\cong(V',\upgam\hs W');
\end{equation}
we apply this to $V'=V$, $V'=\Lambda^2 V$, and $V'=\Lambda^2 V^*$.

\begin{lemma}\label{l:gamma-L}
Let $V=\C^n$ and let $(V,F)$ be a non-degenerate pair as in Construction \ref{con:VF-to-L}.
We have $\dgam L(V,F)\simeq L(V,F)$ if and only if
there exists $g\in\GL(V)$ such that $g_*(\upgam F)=F$.
\end{lemma}

\begin{proof}
By \eqref{e:dgam} and \eqref{e:upgam} we have
\[ \dgam L(V,F)\cong L(\dgam V, \dgam F)\cong L(V, \upgam F).\]
By Corollary \ref{c:same-V} we have $L(V, \upgam F)\simeq L(V,F)$
if and only if there exists $g\in\GL(V)$
such that $g_*(\upgam F)=F$, as desired.
\end{proof}

\begin{corollary}
Let $V=\C^n$ and let $(V,W)$ be a non-degenerate $*$-pair. 
Write $L(V,W)$ for $L(V,F)$ where $F=W^\bot\subset\Lambda^2 V$.
We have $\dgam L(V,W)\simeq L(V,W)$ if and only if
there exists $g\in\GL(V)$ such that $g^*(\upgam\hs W)=W$.
\end{corollary}

\begin{lemma}\label{l:real-form}
Let $V_\R=\R^n$, $V=\C^n$, and let $(V,F)$ be a non-degenerate pair as in Construction \ref{con:VF-to-L}.
The Lie algebra $L(V,F)$ admits a real form
if and only if there exists $g\in\GL(V)$
such that
\begin{equation}\label{e:exist-real}
\upgam(g_*(F))= g_*(F).
\end{equation}
\end{lemma}

\begin{proof}
Assume that there exists $g\in\GL(V)$ satisfying \eqref{e:exist-real}.
Set $F'=g_*(F)$; then $\upgam F'=F'$.
Set $F'_\R=\{x\in F'\ |\ \upgam x=x\}$.
Then $F'_\R\otimes_\R\C=F'$  and $(V_\R,F'_\R)\otimes_\R\C=(V,F')$,
whence $L(V_\R, F'_\R)\otimes_\R\C=L(V,F')$ and  $L(V_\R, F'_\R)$ is a real form of $L(V,F')$.
By Corollary \ref{c:same-V} we have
\[ L(V,F)\cong L(V, g_*(F))=L(V,F').\]
Thus $L(V_\R, F'_\R)$ is also a real form of $L(V,F)$.

Conversely, assume that there exists a real form $L'_\R$ of $L(V,F)$;
then we have an isomorphism $L'_\R\otimes_\R\C\isoto L(V,F)$.
Set  $V'_\R=V_{L'_\R}$, $F'_\R=F_{L'_\R}$ as in Construction \ref{con:L-to-VF},
and write $V'=V'_\R\otimes _\R\C$, $F'=F'_\R\otimes _\R\C$.
Then we have isomorphisms
\[ L(V'_\R,F'_\R)\isoto L'_\R\quad\ \text{and}\quad\ L(V', F')\isoto L'_\R\otimes_\R\C\isoto L(V,F).\]
Choose an isomorphism $\varphi\colon V'_\R\isoto V_\R$
(existing because $V'\simeq V$, whence $\dim V'_\R=\dim V_\R$),
and set $F''_\R=\varphi_*(F'_\R)$;
then $F''_\R\subset \Lambda^2 V_\R$.
Set $F''=F''_\R\otimes_\R\C\subset (\Lambda^2 V_\R)\otimes_\R\C=\Lambda^2 V$; then $\upgam F''=F''$.
It follows from our constructions that
\[ L(V,F'')\simeq L(V',F')\simeq L(V,F),\]
and by Corollary \ref{c:same-V} there exists $g\in\GL(V)$ such that
$F''=g_*(F)$.
Thus
\[\upgam\hs(g_*(F))=\upgam F''=F''=g_*(F),\]
as desired.
\end{proof}

\begin{corollary}
Let $V_\R=\R^n$, $V=\C^n$, and let $(V,W)$ be a non-degenerate $*$-pair.
The Lie algebra $L(V,W)$ admits a real form
if and only if there exists $g\in\GL(V)$
such that
\begin{equation*}
\upgam\hs(g^*(W))= g^*(W).
\end{equation*}
\end{corollary}

\section{A complex Lie algebra admitting no real form}
\label{s:F}

Let $V=\C^6$ with standard basis $e_1,\dots,e_6$, and let $i\in\C$ denote a fixed
square root of $-1$.
Then $\Lambda^2 V\simeq \C^{15}$.
Write
\[j=\SmallMatrix{0&1\\-1&0},\quad J={\rm diag}(j,j,j)\in\GL(6,\C).\]
Then $J^2=-\Id_6\in\GL(6,\C)$ and $J_*^2=\Id_{15}\in\GL\big(\Lambda^2 V\big)$.

\begin{proposition}\label{p:real-orbit}
For $V=\C^6$, let $F\subset \Lambda^2 V$ be a subspace of codimension 4  with the following properties:
\begin{enumerate}
\item[\rm (a)] $\Stab_{\GL(V)}(F)=\{\lambda\cdot\Id_V,\ \lambda\in \C^\times\}$;
\item[\rm (b)] $J_*(\hs\upgam F)=  F$ where $\gamma\in{\rm Gal}(\C/\R)$ denotes the complex conjugation.
\end{enumerate}
Then the orbit of $F$ under $G=\GL(V)$ in the Grassmannian ${\rm Gr}(11,\Lambda^2 V)$
of 11-dimensional subspaces in $\Lambda^2 V$
is stable under $\gamma$, but has no real elements,
that is, there is no $g\in\GL(V)$ such that
\begin{equation}\label{e:real}
\upgam\hs\big (g_*(F)\big)=g_*(F).
\end{equation}
\end{proposition}

\begin{proof}
It immediately follows from (b) that our orbit is $\gamma$-stable.
Assume for the sake of contradiction that there exists  $g$ as in \eqref{e:real}.
Then
\[ (\hs\upgam g\hs)_*\hs(\hs\upgam F)= g_*(F),\]
whence using (b) we obtain that
\[ \upgam g_*\hs\big(J_*^{-1} (F)\big)
=g_*(F)\]
and
\[(g^{-1}\hs \upgam g\hs J^{-1})_*\hs(F)= F.\]
Then it follows from (a) that
\[g^{-1}\hs\upgam g\hs J^{-1} = \lambda\cdot {\rm Id}_V\quad\ \text{for some}\ \, \lambda\in \C^\times\]
and
\[g^{-1}\cdot\upgam g =\lambda J.\]
Applying the operation $x\rightsquigarrow x\cdot \upgam x$ to this equality, we obtain
\[g^{-1}\hs\upgam g\, \upgam g^{-1}\hs g=\lambda\hs\upgam \lambda\hs J^2,\]
that is,
\[\Id_V=-\lambda\hs\upgam\lambda\,\Id_V\hs,\]
whence
\[\lambda\hs\upgam\lambda=-1,\]
which is a contradiction.
\end{proof}

\begin{corollary}\label{c:nd-C6-F}
For a non-degenerate pair $(\C^6,F)$ with properties  (a) and (b)
of Proposition \ref{p:real-orbit}, let $L=L(\C^6,F)$.
Then $\dgam L\simeq L$, but $L$ admits no real form.
\end{corollary}

\begin{proof}
We may identify $\dgam L(\C^6,F)$ with $L(\C^6,\upgam F)$.
By Lemmas \ref{l:gamma-L} and \ref{l:real-form},
it follows from Proposition \ref{p:real-orbit}
that $\dgam L \simeq L$, but there is no real structure on $V=\C^6$
under which the subspace  $F\subset \Lambda^2 V$ is defined over $\R$,
that is, stable under complex conjugation.
Thus, our 10-dimensional two-step nilpotent Lie algebra $L=L(\C^6,F)$ with 4-dimensional center
is isomorphic to its complex conjugate, but admits no real form.
\end{proof}

We consider the specific 11-dimensional subspace $F\subset \Lambda^2 V=\Lambda^2\C^6$,
the orthogonal complement of the specific 4-dimensional subspace
$W=\langle w_1,\dots,w_4\rangle $ in $\Lambda^2 V^*=\Lambda^2 (\C^6)^*$
generated by the four vectors $w_1,\dots,w_4$ of  Theorem \ref{t:main}.

\begin{lemma}\label{l:non-deg}
The pair $(\C^6,F)$ is non-degenerate, whence the $*$-pair $(\C^6,W)$ is non-degenerate.
\end{lemma}

\begin{proof}
Theorem \ref{t:main} explicitly gives the structure constants of the Lie
algebra $L(\C^6,F)$. It is then a straightforward hand calculation to show that
the center of $L(\C^6,F)$ has dimension 4. Proposition \ref{p:bij} shows
that $(\C^6,F)$ is non-degenerate. By Lemma \ref{l:ndp} $(\C^6,W)$ is
non-degenerate as well.
\end{proof}

\begin{lemma}\label{l:(b)}
We have
$J^*(\hs\upgam\, W)= W$, and therefore $F\coloneqq W^\bot$
has property (b) of Proposition \ref{p:real-orbit}.
\end{lemma}

\begin{proof}
It is easy to check the equality in the lemma  by hand or using a computer; we have checked it both ways.
Note that 
\begin{equation*}
J^*\big(\hs\upgam\hs (i w_1)\hs\big)=i w_1,\quad
J^*(\hs\upgam\hs w_\ell)= w_\ell\quad\ \text{for each}\ \, \ell=2,3,4,
\end{equation*}
and $W$ is spanned by the vectors $iw_1, w_2,w_3,w_4$, 
which are fixed by the anti-linear  transformation $w\mapsto J^*(\hs\upgam w\hs)$.
\end{proof}

\begin{theorem}\label{t:Stab}
For $V=\C^6$, the stabilizer of $[W]\in\Gr\big(4,\Lambda^2V^*\big)$ in $\ol G=\PGL(6,\C)$ is trivial.  
In other words,  if $g\in\GL(6,\C)$ satisfies $g^*(W)=W$, then $g$ is a scalar matrix.
\end{theorem}

Theorem \ref{t:Stab} will be proved in Section \ref{s:Stab-W}.

\begin{proof}[Proof of Theorem \ref{t:main} modulo Theorem \ref{t:Stab}]
Take $V=\C^6$, $F=W^\bot\subset \Lambda^2 V$.
The surjection 
$\Lambda^2 V\to \C^4, \ \,\xi\mapsto (\langle w_\ell,\xi\rangle)_\ell\hs$,
has kernel $F$ and identifies $L_\tftf$ with $L(\C^6,F)$.
By Lemma \ref{l:non-deg} the pair $(V,F)$ is non-degenerate.
It follows from Theorem \ref{t:Stab}
that the stabilizer  $\ol G_F$ in $\ol G$ of  $F=W^\bot$ is trivial,
that is, $F$ has property (a) of Proposition \ref{p:real-orbit}.
By Lemma \ref{l:(b)} the subspace  $F$ has property (b) of Proposition \ref{p:real-orbit}.
Now Theorem \ref{t:main} follows from Corollary \ref{c:nd-C6-F}.
\end{proof}

\begin{proof}[Proof of Theorem \ref{t:m} modulo Theorem \ref{t:main}]
For $L$ as in Theorem \ref{t:main},
we must prove that  (i) $L'\coloneqq L\times \C^m$ is isomorphic to its complex conjugate,  but (ii) $L'$ admits no real form.

We prove (i). We have 
\[\dgam L'=\dgam(L\times \C^m)=\dgam L\times\dgam\hs\C^m\simeq L\times\C^m=L',\]
because $\dgam L\simeq L$ by Theorem \ref{t:main}(i), and clearly $\dgam\hs\C^m\cong\C^m$.

We prove (ii).
By Proposition \ref{p:dg-R-C} applied to the field extension $\C/\R$
and the non-degenerate two-step nilpotent  Lie algebra $L$ over $\C$,
if $L'$ admits a real form, then also $L$ must admit a real form,
but this contradicts Theorem \ref{t:main}(ii).
Hence $L'$ admits no real form, as desired.
\end{proof}

\begin{construction}\label{cons:W}
We explain how one can come to the four tensors  $w_1,\dots,w_4$ of \eqref{e:WW} in Theorem \ref{t:main}.
For $1\le j<k\le 6$,
consider the tensors
\[ e^*_{jk}=e_j^*\wedge e_k^*\in \Lambda^2 V_6^*.\]

For $1\le j<k\le 6$ consider the 30 tensors
\[   e_{jk}^*+ J^*(e_{jk}^*)\quad\text{and}\quad i\cdot\big(e_{jk}^*- J^*(e_{jk}^*)\big)\]
where $i^2=-1$.
Of these 30 tensors, choose 15 linearly independent: 
\begin{center}
\renewcommand{\arraystretch}{1.25}
\begin{tabular}{@{}llll@{}}
$f_1=2e^*_{12}$ & $f_2=e^*_{13}+e^*_{24}$ & $f_3=\ii(e^*_{13}-e^*_{24})$ & $f_4=e^*_{14}-e^*_{23}$\\
$f_5=\ii(e^*_{14}+e^*_{23})$ & $f_6=e^*_{15}+e^*_{26}$ & $f_7=\ii(e^*_{15}-e^*_{26})$ & $f_8=e^*_{16}-e^*_{25}$\\
$f_9=\ii(e^*_{16}+e^*_{25})$ & $f_{10}=2e^*_{34}$ & $f_{11}=e^*_{35}+e^*_{46}$ & $f_{12}=\ii(e^*_{35}-e^*_{46})$\\
$f_{13}=e^*_{36}-e^*_{45}$ & $f_{14}=\ii(e^*_{36}+e^*_{45})$ & $f_{15}=2e^*_{56}$ & \\
\end{tabular}
\end{center}
Then $J^*(\hs\upgam f_\vk\hs) = f_\vk$ for all $\vk=1,\dots,15$.

We choose $w_1,\dots,w_4$ to be linear combinations of $f_\vk$ with all coefficients real, 
or with all coefficients purely imaginary.
Then $J^*(\hs\upgam\hs w_\ell\hs) =\pm w_\ell$\hs,  whence 
$\langle J^*(\hs\upgam\hs w_\ell\hs)\rangle =\langle w_\ell\rangle$
for $\ell=1,\dots,4$.
We set $W=\langle w_1,\dots, w_4\rangle\subset  \Lambda^2 V_6^*$.
Then $J^*(\hs\upgam\hs W\hs)=W$.

We choose $w_1,\dots, w_4$ to be linear combinations of $f_\vk$ 
with ``nice''  coefficients and with not too many zero coefficients,
and we check using a computer whether the stabilizer of $W$ in $\gl(6,\C)$ consists of scalars only.
If yes, then this set $w_1,\dots,w_4$ is a candidate for checking
whether the stabilizer of $W$ in $\GL(6,\C)$ consists of scalars only.
This is how one can come to the set $w_1,\dots,w_4$ of Theorem \ref{t:main}.
\end{construction}

\begin{remark} \label{cons:W'}
We explain another method permitting  one to construct many Lie algebras with the required properties.
We choose 60 pseudo-random numbers $a_1,\dots,a_{60}$
taking values $-1,\ 0,\ 1$ with the same probability $1/3$,
for example,
\begin{center}
\begin{tabular}{c c c c c c c c c c c c c c c c  }
$a_1$ – $a_{15}$:   &1 &0 &$-1$ &0 &1 &$-1$ &$-1$ &$-1$ &$-1$ &1  &0  &1 &$-1$ &0 &1\\
$a_{16}$ – $a_{30}$:&0 &$-1$ &0 &1 &1  &0 &0  &1  &$-1$ &$-1$ &0 &$-1$ &0 &0 &1\\
$a_{31}$ – $a_{45}$:  &0 &1  &0 &1 &0 &0 &0  &$-1$ &$-1$ &1  &0 &$-1$ &0 &1 &1\\
$a_{46}$ – $a_{60}$:  &1 &$-1$ &1 &1 &0 &0 &$-1$ &1  &1  &0  &0 &0  &0 &0 &0
\end{tabular}
\end{center}
and consider the 4 tensors
\[w'_\ell=\sum_{\vk=1}^{15} a_{15(\ell-1)+\vk}\cdot f_\vk\hs, \quad \ell=1,\dots,4.\]
Let $W'\subset \Lambda^2 V_6^*$ denote the subspace generated by $w_1',\dots,w_4'$\hs.
Again, we have by construction that $J^*(\hs\upgam\, W'\hs)=W'$.
Our computations show that in this and many similar examples,
the $*$-pair $(V_6,W')$ is non-degenerate with
stabilizer of $W'$ in $\GL(6,\C)$ consisting of scalars only,
and therefore the corresponding Lie algebra $L_{\tftf'}$ 
with $\tftf'=\tftf(w_1',\dots,w_4')$ as in \eqref{e:tf}
is isomorphic to its complex conjugate but admits no real form.
\end{remark}

\section{The stabilizer of the 4-dimensional subspace}
\label{s:Stab-W}

In this section we prove Theorem \ref{t:Stab}.

Using a computer, we can easily show that the {\em Lie algebra} 
of the stabilizer $\ol G_W$ of $W$ in $\ol G=\PGL(6,\C)$ is trivial;
see Implementation \ref{exa:1}
(we used such computations when choosing  our specific $W$).
Therefore, $\ol G_W$ is finite.
Using the Gr\"obner basis algorithm,
we were able to compute the set of involutions in $\ol G_W$: no involutions.
However, a Gr\"obner basis computation of the whole finite group $\ol G_W$
or even of the set of elements of order 3 in $\ol G_W$
did not terminate on the computers available to us.

Instead, we use an idea of the LLM Claude Fable 5. Let $H=G_W$ denote the stabilizer of $W$
in $G=\GL(6,\C)$. Then $H$ naturally acts on $W$.

We identify $\Lambda^2 V^*$ with the space $\Mat_6^\Skew$ 
of skew-symmetric $6\times 6$ matrices.
Namely, we identify a skew-symmetric matrix $A\in \Mat_6^\Skew$ 
with the  skew-symmetric bilinear form
\begin{equation}\label{e:bilinear}
(u,v)\mapsto u\hs A \, v^\TT\in\C
\end{equation}
where $u,v\in\C^6$ are {\em row} vectors and the superscript $^\TT$ denotes the transposed matrix.
Let $g\in\GL(V)=\GL(6,\C)$ act on the right on row vectors $u\in V=\C^6$ by $u\mapsto u g$, 
and act on the left by $u\mapsto u g^{-1}$.
Then under the identification \eqref{e:bilinear},  
the left action of $g$ on $\Lambda^2V^*$ becomes
\begin{equation}\label{eq:action}
A\,\mapsto\, g A g^\TT.
\end{equation}
This is the action denoted by $g^*$ in Sections \ref{s:V-F}--\ref{s:F}; thus $g^*(A)= g A g^\TT$.

The determinant $\det(A)$ of a skew-symmetric matrix $A$ is the square of a polynomial $\Pf(A)$
called the {\em Pfaffian} of $A$; see, for instance, \cite[XV.\S9]{Lang}.
The Pfaffian satisfies the equivariance
\begin{equation}\label{eq:equivariance}
\Pf(g A g^\TT)=\det(g)\,\Pf(A).
\end{equation}
Since  $\Pf(A)^2=\det A$,  a $6\times6$ skew-symmetric matrix
$A$ has $\rk A\le4$ if and only if $\Pf(A)=0$.

In particular, for each $w=w_1,\dots,w_4$, we identify the skew-symmetric bilinear form
\[w=\sum_{k<j}c_{kj}\,e_k^*\wedge e_j^*\in\Lambda^2V^*\]
with the skew-symmetric matrix 
\[A_w=(c_{kj})\in {\rm Mat}_6^\Skew(\C),\quad\ A_w[k,j]=c_{kj}=-A_w[j,k].\]
For $x=(x_1,x_2,x_3,x_4)\in\C^4$, put
\[A(x)=x_1A_{w_1}+x_2A_{w_2}+x_3A_{w_3}+x_4A_{w_4}\hs.\]
Explicitly,
\begin{equation}\label{eq:AAmatrix}
A(x)=\begin{pmatrix}
0 & x_4 & \tfrac{i}{2}x_4 & x_3 & x_2 & -x_2-x_3\\[1pt]
-x_4 & 0 & -x_3 & -\tfrac{i}{2}x_4 & x_2+x_3 & x_2\\[1pt]
-\tfrac{i}{2}x_4 & x_3 & 0 & 0 & -x_2-x_3 & x_1\\[1pt]
-x_3 & \tfrac{i}{2}x_4 & 0 & 0 & x_1 & -x_2-x_3\\[1pt]
-x_2 & -x_2-x_3 & x_2+x_3 & -x_1 & 0 & 0\\[1pt]
x_2+x_3 & -x_2 & -x_1 & x_2+x_3 & 0 & 0
\end{pmatrix}.
\end{equation}

\begin{lemma}
With $f(x):=\Pf(A(x))$ one has
\begin{equation}\label{eq:f}
\begin{aligned}
&f(x)=x_1^2x_4-i\,x_1x_2x_4-2x_2^2x_3-x_2^2x_4-2x_2x_3^2-2x_2x_3x_4-x_3^2x_4
=q_2(x_1,x_2,x_3)\,x_4+q_3(x_2,x_3),\\
&\text{where}\quad\ 
q_2(x_1,x_2,x_3)=x_1^2-i\,x_1x_2-x_2^2-2x_2x_3-x_3^2,\quad\ 
q_3(x_2,x_3)=-2x_2x_3(x_2+x_3).
\end{aligned}
\end{equation}
\end{lemma}

\begin{proof}
A Magma computation using the Magma function \texttt{Pfaffian};
see Implementation \ref{exa:2}.  
\end{proof}

We consider the group $\GL(W)$, which we identify with $\GL(4,\C)$
using the basis $w_1,\dots,w_4$ of $W$.
Since by \eqref{eq:equivariance} the group $\GL(V)=\GL(6,\C)$
acting by \eqref{eq:action} preserves the Pfaffian up to a nonzero scalar,
so does the image $H|_W\subset \GL(W)$  of $H=\Stab_{\GL(V)}(W)$.
Thus
\begin{equation}\label{e:H|_W}
H|_W\subseteq\Stab_{\GL(W)}(\C^\times\scdot f) \quad\ \text{where $f$ is the cubic polynomial \eqref{eq:f}.}
\end{equation}

\begin{proposition}\label{p:Stab(f)}
$\Stab_{\GL(W)}(\C^\times\scdot f)=(\C^\times\scdot\Id_4)\cup(\C^\times\scdot \phi)$,
where
\begin{equation*}
\phi =
\begin{pmatrix}
-1 &i&0&0\\
0&1&0&0\\
0&0&1&0\\
0&0&0&1
\end{pmatrix}.
\end{equation*}
\end{proposition}

\begin{proof}
A computer computation using the Gr\"obner basis algorithm;
see Implementation \ref{exa:3}.  
\end{proof}

Now  let $x\in\C^4$, and let $A(x)$ be the skew-symmetric matrix as in \eqref{eq:AAmatrix}.
Let $M=\Id_4$ or $M=\phi\in\GL(4,\C)$.
Let $G=\GL(6,\C)$.
We wish to find all $g\in G$ such that
\begin{equation}\label{e:M}
g A(x) g^\TT=A(Mx)\quad\ \text{up to scalar, for all}\ \, x\in\C^4.
\end{equation}
From this non-linear condition, following an idea of the LLM Claude Fable 5, we pass to linear ones.

Since $A(x)$ and $A(Mx)$ are skew-symmetric,
the annihilator of $\im A$ is $\ker A$: indeed, we have
\[ n^\TT A=-(An)^\TT.\]
Note that \eqref{e:M} implies
\[\im(g\cdot A(x)\hs)=\im A(Mx),\]
because $g^\TT$ is invertible.
Hence, we have
\begin{equation}\label{e:lin}
n^\TT\cdot g\cdot u=0\quad\ \text{for all}\ \, n\in \ker A(Mx),\ u\in \im A(x),\ x\in \C^4,
\end{equation}
where $n$ and $u$ are column vectors. 
We obtain {\em linear} equations for $g$. We solve them using a computer;
see Implementation \ref{exa:4}.

We obtain the following results:
for $M=\Id_4$ we have $g\in \C^\times\scdot\hs \Id_6$, 
and for $M=\phi$ the only solution is 0, hence no invertible $g$.
Conversely, any $g\in \C^\times\scdot\hs \Id_6$ clearly satisfies \eqref{e:M}.
Thus we obtain:

\begin{lemma}\label{l:M}
The preimage of $(\C^\times\scdot\hs\Id_4)\cup(\C^\times\scdot \phi)\subset \GL(W)$
in $\Stab_{\GL(V)}(W)$ is $\C^\times\scdot\hs \Id_6$.
\end{lemma}

\begin{proof}[Proof of Theorem \ref{t:Stab}]
By \eqref{e:H|_W}, we have $H|_W\subseteq \Stab_{\GL(W)}(\C^\times\scdot f)$.
Using  Proposition \ref{p:Stab(f)}, we obtain that
\[H|_W\hs\subseteq\hs\Stab_{\GL(W)}(\C^\times\scdot f)\hs\subseteq\hs (\C^\times\scdot\hs\Id_4)\cup(\C^\times\scdot \phi).\]
Using Lemma \ref{l:M}, we obtain that $H\subseteq \C^\times\scdot\hs \Id_6$.
Thus $H= \C^\times\scdot\hs \Id_6$\hs, as desired.
\end{proof}

This completes the proofs of  Theorems \ref{t:main} and \ref{t:m}; see Section \ref{s:F}.

\begin{remark}
It is known that the generic stabilizer for the action of
$\PGL_6$ on the Grassmannian $\Gr\big(4,\Lambda^2 V_6\big)$ is trivial.
This follows, for example, from \cite[Table 1.4]{GL24}.
Indeed, that result shows that the stabilizer for the algebraic group is generically trivial,
and this is sufficient in characteristic zero.
In positive characteristic, one wants to show triviality of the generic
stabilizer as a group scheme.  Given the result about the algebraic group, it suffices to
check that the Lie algebra stabilizer is generically $0$.  This can be
checked at a single point  (presumably in characteristic not $2$ the
same point will work).
As in \cite{GG20}, if we are in characteristic $p$, it suffices to
check elements $x$ in the Lie algebra satisfying $x^p=x$ or $x^p=0$
and show that the dimension of the $\PGL_6$ orbit of $x$ plus the dimension of the
variety of fixed points of $x$ on the Grassmannian is strictly less
than $44$ (the dimension of the Grassmannian).
This is a straightforward linear algebra computation by considering
the various possibilities for the similarity class of $x$.

Unfortunately, this does not give rise to a Lie algebra example as in
characteristic zero.  Indeed, by the standard cohomology argument and
Lang's theorem, if $L$ is a finite-dimensional Lie algebra in characteristic $p$ with
connected automorphism group and $L$ is isomorphic to its $q$-power
Frobenius twist, then $L$ is defined over the field of $q$ elements.
\end{remark}

\subsection*{Acknowledgements.}
We thank Boris Kunyavski\u{\i} for suggesting Remarks \ref{r:K1} and \ref{r:K2}, and 
we thank Taeyeoup Kang for helpful discussions and email correspondence.
We are grateful to Skip Garibaldi for putting our question to the LLM Claude Fable 5.

Claude Fable 5 gave an autonomous proof of Theorem \ref{t:Stab} 
on which the proof of Theorem \ref{t:main} is based,
but in this paper we give a full proof of this theorem 
using only Claude's  {\em ideas}, not its proof.

The final version of this paper was proofread by Claude, 
and while proofreading, Claude suggested adding
Propositions \ref{p:nd} and \ref{p:dg-R-C},
which together gave a better proof of Theorem \ref{t:m}
than our original one.

\appendix

\section{Program code}
\label{app:code}

In this appendix we provide the code of the programs used in our computer calculations. 
We use the computer algebra systems {\sf GAP}4 \cite{GAP4} along with
its package {\sf SLA} \cite{sla}, and {\sc Magma} \cite{magma}.

\begin{implementation}\label{exa:1}
Let $W$ be the space in $\bigwedge^2 V_6^*$ spanned by $w_1,w_2,w_3,w_4$
of Theorem \ref{t:main}. Here we give a {\sf GAP} calculation that shows
that the Lie algebra of the stabilizer of $W$ in $\GL(6,\C)$ is spanned
by the identity. In our {\sf GAP} sessions we assume that the
{\sf SLA} package is loaded.

First we construct the Lie algebra $\gl(6,\C)$ (which we denote \verb+L+)
and the module $\bigwedge^2 V_6^*$:
\begin{verbatim}
gap> L:= FullMatrixLieAlgebra( CF(4), 6 );;
gap> V6:= LeftAlgebraModule( L, function(x,v) return x*v; end, CF(4)^6 );;
gap> DV6:= DualAlgebraModule( V6 );;
gap> Wed6:= ExteriorPowerOfAlgebraModule( DV6, 2 );;
\end{verbatim}

Next we define the elements $w_1,w_2,w_3,w_4$. By inspecting the basis
of \verb+Wed6+ one sees that the elements below are correct. Here we do not
perform this verification.

\begin{verbatim}
gap> vv:=BasisVectors( Basis(Wed6) );;
gap> w1:= vv{[3,4]}*[1,1];;
gap> w2:= vv{[2,5,7,8,11,12]}*[-1,-1,1,1,-1,1];;
gap> w3:= vv{[2,5,8,10,11,13]}*[-1,-1,1,-1,-1,1];;
gap> w4:= vv{[9,14,15]}*[-E(4)/2,E(4)/2,1];;
\end{verbatim}

Unfortunately, currently in {\sf GAP} there is no function for computing
the stabilizer of a subspace. Therefore, we set up the equations by hand:

\begin{verbatim}
gap> n:= Dimension( L );;
gap> x:= BasisVectors( Basis(L) );;
gap> w:= [w1,w2,w3,w4];;
gap> a:= List( [1..n], i -> [] );;
gap> for i in [1..n] do for j in [1..4] do 
> a[i][j]:= Coefficients( Basis(Wed6), x[i]^w[j] ); od; od;
gap> ww:= List( w, v -> Coefficients( Basis(Wed6), v ) );;

gap> eqns:=[];;
gap> for j in [1..4] do for k in [1..15] do
> eq:= List( [1..n+16], r -> 0 );;
> for i in [1..n] do eq[i]:= a[i][j][k]; od;
> for s in [1..4] do eq[n+s+4*(j-1)]:= eq[n+s+4*(j-1)]+ww[j][k]; od;
> Add( eqns, eq );
> od; od;
\end{verbatim}

We solve the equations. For each solution we take the first \verb+n+
coordinates and form the subspace of \verb+L+ spanned by them:

\begin{verbatim}
gap> sol:= NullspaceMat( TransposedMat( eqns ) );;
gap> ss:= List( sol, s -> s{[1..n]} );;
gap> stab:= Subalgebra( L, List( ss, s -> s*Basis(L) ) );;
gap> Dimension( stab );
1
gap> Display(Basis(stab)[1]);
[ [  1,  0,  0,  0,  0,  0 ],
  [  0,  1,  0,  0,  0,  0 ],
  [  0,  0,  1,  0,  0,  0 ],
  [  0,  0,  0,  1,  0,  0 ],
  [  0,  0,  0,  0,  1,  0 ],
  [  0,  0,  0,  0,  0,  1 ] ]
\end{verbatim}
\end{implementation}  

\begin{implementation}\label{exa:2}
The following piece of {\sc Magma} code computes the Pfaffian of the matrix in
\eqref{eq:AAmatrix}.

\begin{verbatim}
> F<i>:= CyclotomicField(4);
> P<x1,x2,x3,x4>:= PolynomialRing(F,4);
> m:= [[0 , x4 , i/2*x4 , x3 , x2 , -x2-x3],
> [-x4 , 0 , -x3 , -i/2*x4 , x2+x3 , x2],
> [-i/2*x4 , x3 , 0 , 0 , -x2-x3 , x1],
> [-x3 , i/2*x4 , 0 , 0 , x1 , -x2-x3],
> [-x2 , -x2-x3 , x2+x3 , -x1 , 0 , 0],
> [x2+x3 , -x2 , -x1 , x2+x3 , 0 , 0]];
> Pfaffian( Matrix( m ) );
x1^2*x4 - i*x1*x2*x4 - 2*x2^2*x3 - x2^2*x4 - 2*x2*x3^2 - 2*x2*x3*x4 - x3^2*x4
\end{verbatim}
\end{implementation}  

\begin{implementation}\label{exa:3}
The following piece of code computes polynomial equations describing\newline
$\Stab_{\GL(W)}(\C^\times\scdot f)$ with $f$ as in \eqref{eq:f}. 
\begin{verbatim}
> F<i>:= CyclotomicField(4);
> P<r,d,a11,a12,a13,a14,a21,a22,a23,a24,a31,a32,a33,a34,a41,a42,a43,a44>:= 
> PolynomialRing( F, 18 );
> R<x1,x2,x3,x4>:= PolynomialRing( P, 4 );
> f:= x1^2*x4 - i*x1*x2*x4 - 2*x2^2*x3 - x2^2*x4 - 2*x2*x3^2 - 2*x2*x3*x4 - x3^2*x4;
> M:= Matrix( R, 4, 4, [ P.i : i in [3..18] ] );
> v:= Vector( [x1,x2,x3,x4] );
> vals:= Eltseq( v*Transpose(M) );
> h:=Evaluate( f, vals ) -r*f;
> eqns:= Coefficients(h);
> eqns cat:= [ Determinant( Matrix( 4, 4, [ P.i : i in [3..18] ] ))*d-1 ];
> eqns:= Reduce(eqns);
> time gb:= GroebnerBasis(eqns);
Time: 260.910
> gb;
[
    r - a44^3,
    d*a12*a44^3 - 1/2*i*d*a44^4 + 1/2*i,
    d*a44^5 - 2*i*a12 - a44,
    a11 - 2*i*a12 - a44,
    a12^2 - i*a12*a44,
    a13,
    a14,
    a21,
    a22 - a44,
    a23,
    a24,
    a31,
    a32,
    a33 - a44,
    a34,
    a41,
    a42,
    a43
]
\end{verbatim}
\end{implementation}

\begin{implementation}\label{exa:4}
Here we give code for solving the linear equations \eqref{e:lin}. Note that
\eqref{e:lin} yields a number of linear equations for each $x\in \C^4$. So
a priori we have an infinite number of equations. We deal with this by
choosing six elements $x\in \C^4$ such that the rank of $A(x)$ is less than
6. The corresponding equations then give the desired result.

First in {\sf GAP} we define the following functions.
\begin{verbatim}
mat:= function( vec )

        local x1, x2, x3, x4, i;

        x1:= vec[1]; x2:= vec[2]; x3:= vec[3]; x4:= vec[4];
        i:= E(4);

        return [[0 , x4 , i/2*x4 , x3 , x2 , -x2-x3],
                [-x4 , 0 , -x3 , -i/2*x4 , x2+x3 , x2],
                [-i/2*x4 , x3 , 0 , 0 , -x2-x3 , x1],
                [-x3 , i/2*x4 , 0 , 0 , x1 , -x2-x3],
                [-x2 , -x2-x3 , x2+x3 , -x1 , 0 , 0],
                [x2+x3 , -x2 , -x1 , x2+x3 , 0 , 0]];

end;

solns:= function(M)

        local xx, eqq, x, AMx, nn, Ax, uu, n, u, eqn, k, l,
              sol, res, s, g;

        xx:= [ [1,0,0,0], [0,1,0,0], [0,0,1,0], [0,0,0,1],
               [1,1,0,0], [1,0,1,0] ];
        eqq:= [ ];
        for x in xx do
            AMx:= mat( M*x );
            nn:= NullspaceMat( TransposedMat(AMx) );
            Ax:= mat( x );
            uu:= TransposedMat( Ax );
            for n in nn do
                for u in uu do
                    eqn:= [];
                    for k in [1..6] do
                        for l in [1..6] do
                            eqn[l+(k-1)*6]:= n[k]*u[l];
                        od;
                    od;
                    Add( eqq, eqn );
                od;
            od;
        od;

        sol:= NullspaceMat( TransposedMat(eqq) );
        res:= [ ];
        for s in sol do
            g:= List( [1..6], q -> [ ] );
            for k in [1..6] do
                for l in [1..6] do
                    g[k][l]:= s[l+(k-1)*6];
                od;
            od;
            Add( res, g );
        od;

        return res;

end;
\end{verbatim}

Then in {\sf GAP}, after having read these functions, we do
\begin{verbatim}
gap> M:= IdentityMat(4);;
gap> solns(M);
[ [ [ 1, 0, 0, 0, 0, 0 ], [ 0, 1, 0, 0, 0, 0 ], [ 0, 0, 1, 0, 0, 0 ], 
    [ 0, 0, 0, 1, 0, 0 ], [ 0, 0, 0, 0, 1, 0 ], [ 0, 0, 0, 0, 0, 1 ] ] ]
gap> M:= [[-1,E(4),0,0],[0,1,0,0],[0,0,1,0],[0,0,0,1]];;
gap> solns(M);
[  ]
\end{verbatim}
\end{implementation}


\end{document}